\documentclass[a4paper,10pt]{article}

\usepackage{hyperref}
\usepackage[english]{babel}
\usepackage{inputenc}
\usepackage[sumlimits]{amsmath}
\usepackage{amsthm}
\usepackage{amstext}
\usepackage{fancyhdr}
\usepackage{amssymb}
\usepackage{mathrsfs}
\usepackage[all]{xy}
\usepackage{enumerate}

\newtheorem{theorem}{Theorem}
\newtheorem{lemma}{Lemma}
\newtheorem{proposition}{Proposition}
\newtheorem{corollary}{Corollary}

\newtheorem{remark}{Remark}

\DeclareMathOperator{\N}{\mathbb{N}}

\DeclareMathOperator{\Z}{\mathbb{Z}}

\DeclareMathOperator{\Com}{C}

\DeclareMathOperator{\PPP}{\mathbb{P}}
\DeclareMathOperator{\Exc}{TExc}
\DeclareMathOperator{\Fix}{Fix}
\DeclareMathOperator{\NFix}{NFix}

\title{Kazhdan--Lusztig polynomials of boolean elements\thanks{This paper is part of the author's Ph.D. thesis written under the direction of Prof. F. Brenti at the Univ. ``la Sapienza" of Rome, Italy.}}
\author{Pietro Mongelli\thanks{Universit\`a "Sapienza" di Roma, Roma,  Italy, e-mail:mongelli@mat.uniroma1.it}}
\date{}
\begin{document}
\setcounter{page}{1}

\maketitle


{\bf Abstract.}\\
We give  closed combinatorial product formulas for Kazhdan--Lusztig poynomials and their parabolic analogue of type $q$ in the case of boolean elements, introduced in [M.\ Marietti, Boolean elements in Kazhdan--Lusztig theory, J.\ Algebra 295 (2006)], in Coxeter groups whose Coxeter graph is a tree.


\section{Introduction}
In their fundamental paper \cite{Kazhdan1979} Kazhdan and Lusztig defined, for every Coxeter group $W$, a family of polynomials, indexed by pairs of elements of $W$, which have become known as the Kazhdan--Lusztig polynomials of $W$ (see, e.\ g., \cite[Chap. 7]{Humphreys1990} or \cite[Chap. 2]{Bjorner2005}). These polynomials play an important role in several areas of mathematics, including the algebraic geometry and topology of Schubert varieties and representation theory (see, e.\ g., \cite[Chap. 5]{Bjorner2005}, and the references cited there). In particular, their coefficients gives the dimensions of the intersection cohomology modules for Schubert varieties (see, e.\ g., \cite{Kazhdan1980}).

In order to find a method for the computation of the dimensions of the intersection cohomology modules corresponding to Schubert varieties in $G/P$, where $P$ is a parabolic subgroup of the Kac-Moody group $G$, in 1987 Deodhar (\cite{Deodhar1987}) introduced two parabolic analogues of  these polynomials which correspond to the roots $x=q$ and $x=-1$ of the equation $x^2=q+(q-1)x$. These parabolic Kazhdan--Lusztig polynomials reduce to the ordinary ones for the trivial parabolic subgroup and are also related to them in other ways (see, e.\ g., Proposition \ref{P:DefQuot} below). Besides these connections the parabolic polynomials also play a direct role in several areas including the theories of generalized Verma modules (\cite{Casian1987}), tilting modules (\cite{Soergel1997}, \cite{Soergel1997bis}) and Macdonald polynomials(\cite{Haglund2005}, \cite{Haglund2005bis}).


The purpose of this work is to give explicit combinatorial product formulas for all (parabolic and ordinary) Kazhdan-Lusztig polynomials indexed by pairs of boolean elements (see Section \ref{S:Def} for the definition) in all Coxeter groups whose Coxeter graph is a tree. Our results show that all such polynomials have nonnegative coefficients, conjectured by Kazhdan and Lusztig \cite{Kazhdan1979}, and give a combinatorial interpretation of them in terms of Catalan numbers and the Coxeter graph of the group.
In the case of classical Weyl groups, this combinatorial interpretation can be restated in terms of excedances and other statistics of (signed) permutations. Our results also confirm a conjecure of Brenti on the parabolic Kazhdan-Lusztig polynomials of type $q$ (see Corollary \ref{C:conjB} below).

The organization of the paper is as follows. In the next section we recall definitions, notation and results that are used in the rest of this work. In Section \ref{S:Preliminary} we give some lemmas about the computation of parabolic Kazhdan--Lusztig polynomials indexed by boolean elements and introduce and illustrate some properties of "Catalan triangle" which will appear in the main result.
In Section \ref{S:KLMain} we state and prove our main result, namely an explicit closed combinatorial formula for all (parabolic and ordinary) Kazhdan--Lusztig polynomials of boolean elements of Coxeter group whose Coxeter graph is a tree. In Section \ref{S:CombKL} we restate the formulas using statistics associated to (signed) permutations for the classical Weyl groups. Finally, in Section \ref{S:Poincare} we use our main result to compute the  intersection homology Poincar\'e polynomials indexed by boolean elements in all Coxeter groups whose Coxeter graphs have at most one vertex with more than two adjacent vertices.

\section{Definitions, notation and preliminaries}
\label{S:Def}
We let $\PPP:=\{1,2,3,\dots\}$, $\N:=\PPP\cup\{0\}$, $\Z:=\N\cup \{-1,-2,\dots\}$. For all $m,n\in \Z$, $m\le n$ we set $[m,n]:=\{m,m+1,\dots,n\}$ and $[n]:=[1,n]$. Given a set $A$ we denote by $\# A $ its cardinality.

We follow \cite[Chap. 3]{Stanley1986} for poset notation and terminology. In particular, given a poset $(P,\le)$ and $u,v\in P$ we let $[u,v]:=\{w\in P\vert u\le w\le v\}$ and call this an \emph{interval} of $P$. We say that $v$ \emph{covers} $u$, denoted $u \triangleleft v$ (or, equivalently, that $u$ is \emph{covered} by $v$) if $\#[u,v]=2$.

We follow \cite{Humphreys1990} for general Coxeter groups notation and terminology. Given a Coxeter system $(W,S)$ and $u\in W$ we denote by $l(u)$ the length of $u$ in $W$, with respect to $S$, i.\ e.\ the minimal length of words $s_{i_1}\cdots s_{i_k}=u$ whose alphabet is $S$ (such minimal words are called reduced). Given $u,v\in W$ we denote by $l(u,v)=l(v)-l(u)$. We let $D_R(u):=\{s\in S\vert l(us)<l(u)\}$ the set of the right descents of $u$, $D_L(u):= \{s\in S\vert l(su)<l(u)\}$ the set of the left descents of $u$ and we denote by $\epsilon$ the identity of $W$. Given $J\subseteq S$ we let $W_J$ the parabolic subgroup generated by $J$ and 
\begin{equation}
\label{E:ParabolicQuotient}
W^J :=\{u\in W\vert l(su)>l(u) \text{ for all } s\in J\}
\end{equation}
Note that $W^\emptyset=W$.  If $W_J$ is finite, then we denote by $w_0(J)$ its longest element. We will always assume that $W^J$ is partially ordered by \emph{Bruhat order}. Recall (see e.g. \cite[Chap. 5.9 and 5.10]{Humphreys1990}) that this means that $x\le y$ if and only if for one reduced word of $y$ (equivalently for all) there exists a subword that is a reduced word of $x$. Given $u,v\in W^J$, $u\le v$ we let
$$
[u,v]^J:=\{w\in W^J\vert u\le w\le v\},
$$
and $[u,v]:=[u,v]^\emptyset$.

For $J\subseteq S$, $x\in \{-1,q\}$ and $u,v\in W^J$ we denote by $P_{u,v}^{J,x}(q)$ the parabolic Kazhdan--Lusztig polynomials in $W^J$ of type $x$ (we refer the reader to \cite{Deodhar1987} for the definitions of these polynomials, see also Proposition \ref{P:DefQuot} below). We denote by $P_{u,v}(q)$ the ordinary Kazhdan--Lusztig polynomials.

The following result is due to Deodhar, and we refer the reader to \cite{Deodhar1987} for its proof. For $u,v\in W^J$ let $\mu_{J,q}(u,v)$ be the coefficient of $q^{\frac{1}{2}(l(u,v)-1)}$ in $P_{u,v}^{J,q}(q)$ (so $\mu_{J,q}(u,v)=0$ when $l(v)-l(u)$ is even). It is well known that if $u,v\in W^J$ then $\mu_{J,q}(u,v)=\mu(u,v)$, the coefficient of $q^{\frac{1}{2}(l(u,v)-1)}$ in $P_{u,v}(q)$ (see Corollary \ref{C:MuEqual} below).

\begin{proposition}
\label{P:KLFromDesc}
Let $(W,S)$ be a Coxeter system, $J\subseteq S$, and $u,v\in W^{J},u\le v$. Then for each $s\in D_R(v)$ we have that
\begin{equation}
\label{E:KLFromDesc}
P_{u,v}^{J,q}(q)=\widetilde P_{u,v}-\widetilde M_{u,v}
\end{equation}
where
$$
\widetilde P_{u,v}=\left \{ \begin{array}{ll} 
P_{us,vs}^{J,q}+q P_{u,vs}^{J,q} & \text{if } us<u;\\
qP_{us,vs}^{J,q}+ P_{u,vs}^{J,q} & \text{if } u<us\in W^J;\\
0 & \text{if } u<us\not\in W^J.
\end{array}\right.
$$
and
$$
\widetilde M_{u,v}=\sum_{u\le w<vs\vert ws<w} \mu(w,vs)q^{\frac{l(w,v)}{2}}P_{u,w}^{J,q}(q).
$$
\end{proposition}

The parabolic Kazhdan--Lusztig polynomials are related to their ordinary counterparts in several ways, including the following one, which may be taken as their definition in most cases.
\begin{proposition}
\label{P:DefQuot}
Let $(W,S)$ be a Coxeter system, $J\subseteq S$ and $u,v\in W^J$. Then we have that 
$$
P_{u,v}^{J,q}(q)=\sum_{w\in W_J} (-1)^{l(w)} P_{wu,v}(q).
$$
Moreover, if $W_J$ is finite, then
$$
P_{u,v}^{J,-1}(q)=P_{w_0(J)u,w_0(J)v}(q).
$$
\end{proposition}
A proof of this result can be found in \cite{Deodhar1987} (see Proposition 3.4, and Remark 3.8).
Since for all $w\in W_J$ and $u\in W^J$ we have $l(wu)=l(w)+l(u)$ by \cite[Proposition 2.4.4]{Bjorner2005}, then the degree of $P_{wu,v}(q)$ in Proposition \ref{P:DefQuot} is less than $\frac{1}{2}(l(u,v)-1)$ except when $w=\epsilon$. Therefore we have
\begin{corollary}
\label{C:MuEqual}
For any $J\subseteq S$ and $u,v\in W^J$ we have
$$
\mu_{J,q}(u,v)=\mu(u,v).
$$
\end{corollary}
The following result is probably known, but for lack of an adequate reference we provide its proof here.
\begin{proposition}
\label{P:PropertiesKL}
Let $(W,S)$ a Coxeter system and $J\subseteq S$. Let $u,v\in W^J$ and $s\in D_R(v)$. 
\begin{enumerate}
	\item[a)] If $us\not \in W^J $ then $P_{u,v}^{J,q}(q)=0$;
	\item[b)] if $us\in W^J$ then $P_{us,v}^{J,q}(q)=P^{J,q}_{u,v}(q)$; 
	\item[c)] if $\mu(u,v)\ne 0$ then $D_R(v)\subseteq D_R(u)$ and $D_L(v)\subseteq D_L(u)$.
\end{enumerate}
\end{proposition}
\begin{proof}
If $us\not \in W^J$ then by Proposition \ref{P:KLFromDesc} we have 
$$
P_{u,v}^{J,q}(q)=-\sum_{u\le w<vs; ws<w}\mu(w,vs)q^{\frac{l(w,v)}{2}}P^{J,q}_{u,w}(q)
$$
The sum may be empty or we can apply induction on $l(v)-l(u)$ and have $P_{u,w}^{J,q}(q)=0$. In both cases $P_{u,v}^{J,q}(q)=0$.
For $b)$ use the same arguments as in the proof of  \cite[Proposition 5.1.8]{Bjorner2005}. For the first part of $c)$ use $a)$, $b)$  and the property that $P_{u,v}^{J,q}(q)$ has maximal degree. For the second part of $c)$ use the identity $P_{u,v}(q)=P_{u^{-1},v^{-1}}(q)$ (see \cite[Exercise 5.12]{Bjorner2005}) and Corollary \ref{C:MuEqual}.
\end{proof}
In the rest of the paper we will consider parabolic Kazhdan--Lusztig polynomials of type $q$. Therefore we will write $P_{u,v}^J$ instead of $P_{u,v}^{J,q}$.

Let $(W,S)$ be any Coxeter system and $t$ be a reflection in $W$. Following Marietti (\cite{Marietti2002}, \cite{Marietti2006} and \cite{Marietti2010}), we say that $t$ is a \emph{boolean reflection} if it admits a {\em boolean expression}, which is, by definition, a reduced expression of the form $s_1\cdots s_{n-1}s_ns_{n-1}\cdots s_1$ with $s_k\in S$, for all $k\in\{1,\dots n\}$ and $s_i\ne s_j$ if $i\ne j$. We say that $u\in W$  is a {\em boolean element} if $u$ is smaller than a boolean reflection in the Bruhat order. Let $\overline v$ be a reduced word of a boolean element and $s\in S$, we denote by $\overline v(s)$ the number of occurrences of $s$ in $\overline v$.

Given a Coxeter system $(W,S)$, the Coxeter graph of $W$ is a graph whose vertex set is $S$ and two vertices $s$, $s'$ are joined by an edge if $ss'\ne s's$. We label this edge with $m(s,s')$, the smallest positive integer  such that $(ss')^{m(s,s')}=\epsilon$ ($m(s,s')=\infty$ if there is no such integer). We say that $W$ is a \emph{ tree-Coxeter group} if its Coxeter graph is a tree.

\section{Preliminary results}
\label{S:Preliminary}
In this section we give some preliminary lemmas which are needed to prove the main theorem in the next section.
For any generator $s_i\in S$  we denote by $\Com(s_i)$ the subset of $S$ which contains all elements commuting with $s_i$ different from $s_i$, and by $S^i=S\setminus\{s_i\}$. 

\begin{lemma}
\label{L:11}
Let $u,v\in W^J$ such that $s_iu,s_iv\in W_{S^i}^J$ (i.\ e.\ there exist reduced words for $u,v$ starting with $s_i$ and with no other occurrences of $s_i$). Then 
$$
P_{u,v}^{J}=P_{s_iu,s_iv}^{J\cap \Com(s_i)}.
$$
\end{lemma}
\begin{proof}
The statement is trivial if $l(v)=1$. Suppose that $l(v)>1$. Then there exists $s_j\in D_R(v)$, $j\ne i$. Note that  for any $w\in W$ with $s_iw\in W_{S^i}$ we have that $D_L(w)\subseteq \{s_i\}\cup (S\cap \Com(s_i))$, more precisely $D_L(w)=\{s_i\}\cup (D_L(s_iw)\cap \Com(s_i))$. Therefore  $us_j\in W^J$ if and only if $s_ius_j\in W^{J\cap \Com(s_i)}$. In this case, by Proposition \ref{P:KLFromDesc} we have
\begin{align*}
P_{u,v}^{J}=&q^c P^{J}_{us_j,vs_j}+q^{1-c}P^{J}_{u,vs_j}-\sum_{{u\le w\le vs_j} \atop {ws_j<w}}\mu(w,vs_j)q^{\frac{l(w,vs_j)}{2}}P^{J}_{u,w}\\
=&q^c P^{J\cap \Com(s_i)}_{s_ius_j,s_ivs_j}+q^{1-c}P^{J\cap \Com(s_i)}_{s_iu,s_ivs_j}+\\
&-\sum_{{s_iu\le s_iw\le s_ivs_j} \atop {s_iws_j<s_iw}}\mu(s_iw,s_ivs_j)q^{\frac{l(s_iw,s_ivs_j)}{2}}P^{J\cap \Com(s_i)}_{s_iu,s_iw}=P_{s_iu,s_iv}^{J\cap \Com(s_i)}
\end{align*}
by induction, where $c$ is $0$ or $1$. The equalities hold since the map from $[u,v]^J$ to $[s_iu,s_iv]^{J\cap \Com(s_i)}$ given by left-multiplication by $s_i$ is an isomorphism of posets.
\end{proof}

\begin{lemma}
\label{L:10}
Let $u,v\in W^J$ be such that $u,s_iv\in W_{S^i}$ (i.\ e.\ there are no occurrences of $s_i$ in any reduced expression of $u$ and $s_iv$). Then
$$
P_{u,v}^{J}=\left\{\begin{array}{ll}P_{u,s_iv}^{J} & \text{if $s_iv\in W^J$}\\ 0 & \text{otherwise}\end{array}\right.
$$
\end{lemma}

\begin{proof}
If $l(v)=1$ there is nothing to prove. Let we suppose $l(v)>1$ and let $s_j\in D_R(v)$, $s_j\ne s_i$. If $us_j\not\in W^J$ the claim is trivial by Proposition \ref{P:PropertiesKL}. Therefore we may assume $us_j\in W^J$.

Suppose that $s_iv\in W^J$. Then by Proposition \ref{P:KLFromDesc} we get

\begin{align*}
P_{u,v}^{J}=&q^c P^{J}_{us_j,vs_j}+q^{1-c}P^{J}_{u,vs_j}-\sum_{{u\le w\le vs_j} \atop {ws_j<w}}\mu(w,vs_j)q^{\frac{l(w,vs_j)}{2}}P^{J}_{u,w}\\
=&q^c P^{J}_{us_j,s_ivs_j}+q^{1-c}P^{J}_{u,s_ivs_j}-\sum_{{u\le s_iw\le s_ivs_j} \atop {s_iws_j<s_iw}}\mu(s_iw,s_ivs_j)q^{\frac{l(s_iw,s_ivs_j)}{2}}P^{J}_{u,s_iw}\\ =&P_{u,s_iv}^{J}
\end{align*}
where $c$ is $0$ or $1$.
The equalities hold  by induction on $l(vs_j)$: if $w\in  W_{S^i}$ then $\mu(w,vs_j)$ is $0$ since by induction either $P_{w,vs_j}^{J}=0$ or $P^{J}_{w,vs_j}=P^{J}_{w,s_ivs_j}$ and therefore $P_{w,vs_j}^{J}$ does not  have the maximum degree. Otherwise, if $s_iw\notin W^J$ then $P_{u,w}^{J}=0$ by induction, else $P_{u,w}^{J}=P_{u,s_iw}^{J}$ and $\mu(w,vs_j)= \mu(s_iw,s_ivs_j)$ by Lemma  \ref{L:11} and Corollary \ref{C:MuEqual}. 

Finally, if $s_iv\not \in W^J$ (and we may assume also that $s_ivs_j\notin W^J$ except in the case $v=s_is_j$ and $u=\epsilon$ which is trivial) then by induction
$$
P_{u,v}^{J}=-\sum_{{u\le w\le vs_j} \atop {ws_j<w}}\mu(w,vs_j)q^{\frac{l(w,vs_j)}{2}}P^{J}_{u,w}.
$$

Fix $w\in W^J$ with $u\le w\le vs_j$ and $ws_j<w$. We prove that $\mu(w,vs_j)P^{J}_{u,w}=0$.
If $w\in W_{S^i}$ then $\mu(w,vs_j)=0$ by induction. Otherwise, if $s_iw\in W_{S^i}$ then by Lemma \ref{L:11} we have $\mu(w,vs_j)=\mu(s_iw,s_ivs_j)$ but $\mu(s_iw,s_ivs_j)=0$ since both $s_ivs_j\notin W^J$ and  $s_iw\in W_{S^i}$ imply that   $D_L(s_ivs_j)\not\subseteq D_L(s_iw)$ (see $c)$ in Proposition \ref{P:PropertiesKL}).
\end{proof}

We now introduce a family of numbers which are used in the next section.
The \emph{Catalan triangle } is a triangle of numbers formed in the same manner as Pascal's triangle, except that no number may appear on the left of the first element (see \cite[sequence A008313]{OEIS}).
$$
\begin{array}{cccccccccc}
1\\
& 1\\
1 & & 1\\
& 2 & & 1\\
2 & & 3 & & 1\\
& 5 & & 4 & & 1\\
5 & & 9 & & 5 & & 1\\
& 14 & & 14 & & 6 & & 1\\
14 & & 28 & & 20 & & 7 & & 1\\
& 42 & & 48 & & 27 & & 8 & & 1
\end{array}
$$
Let $h\ge 1$. We set
 $$f_h(q)=\sum_{i=0}^{[\frac{h}{2}]} C(h,i)q^{[\frac{h}{2}]-i}$$
  where $[h]$ denotes the integer part of $h$ and $C(h,i)$ is the $i$-th number in the $h$-th row (here we start the enumeration from $0$). For example $f_4(q)=2q^2+3q+1$; $f_7(q)=14q^3+14q^2+6q+1$. We denote by $\mu(f_h(q))$ the coefficient of $q^{\frac{h}{2}}$ in $f_h(q)$. Therefore $\mu(f_h(q))=0$ if $h$ is odd. Then we have the following easy result, whose proof we omit.
\begin{lemma}
\label{L:Propf}
For all $h\ge 0$, 
$$
f_h(q)(1+q)-\mu(f_h(q))q^{\frac{h}{2}+1}=f_{h+1}(q).
$$
\end{lemma}
Note that in the first column we find the classical Catalan numbers (see \cite[sequence A008313]{OEIS} for details).

\section{Parabolic Kazhdan--Lusztig polynomials}
\label{S:KLMain}
Let $(W,S)$ be a tree-Coxeter group. Let $t=s_1\cdots s_{n-1}s_ns_{n-1}\cdots s_1$ be a boolean reflection. Consider the Coxeter graph $G$ and  represent it as a rooted tree with root the vertex corresponding to the generator $s_n$. In this paper all the roots will be depicted on the right of their graphs.  In Figure \ref{F:CoxeterGraphDtilde} we give the Coxeter graph of the affine Weyl group $\widetilde D_{11}$.

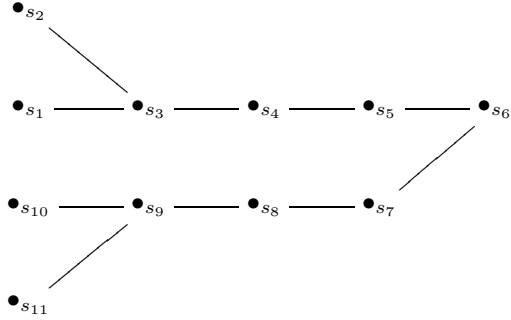
\begin{figure}[htb]
$$
\xymatrix{
 \bullet_{s_2} \ar@{-}[dr]\\
\bullet_{s_1}\ar@{-}[r] &\bullet_{s_3} \ar@{-}[r] & \bullet_{s_4}\ar@{-}[r] &\bullet_{s_5}\ar@{-}[r] & \bullet_{s_6}\\ 
\bullet_{s_{10}}\ar@{-}[r] &\bullet_{s_9} \ar@{-}[r] & \bullet_{s_8}\ar@{-}[r] &\bullet_{s_7}\ar@{-}[ur] \\ 
\bullet_{s_{11}}\ar@{-}[ur]
}
$$
\caption{The Coxeter graph of $\widetilde D_{11}$ with root $s_6$, corresponding to reflection $t=s_1s_2\cdots s_5s_{10}s_{11}s_9s_8s_7s_6s_7s_8s_9s_{11}s_{10}s_5\dots s_2s_1$.}  \label{F:CoxeterGraphDtilde} 
\end{figure}

According to such rooted graph we say that  $s_j$ is on the right (respectively on the left) of $s_i$ if and only if there exists an edge joining them and the only path joining $s_i$ to $s_n$ contains $s_j$.

Let $u,v\in W$ be such that $u,v\le t$. Let $\overline u,\overline v$ be the unique reduced expressions of $u,v$ satisfying the following properties
\begin{itemize}
	\item $\overline v$ is a subword of $s_1\cdots, s_{n-1}s_{n}s_{n-1}\cdots s_1$ and if $i$ is such that $\overline v(s_i)=1$ and  $\overline v(s_j)=0$, where $s_j$ is the only element on the right of $s_i$, then we choose the subword with $s_i$ in the leftmost admissible position;
	\item $\overline u$ is a subword of $\overline v$ and if $i$ is such that $\overline v(s_i)=1$ and  $\overline v(s_j)=0$, we apply the same above rule.
\end{itemize}

Here we give an example. Let $t=s_1s_2\cdots s_5s_{10}s_{11}s_9s_8s_7s_6s_7s_8s_9s_{11}s_{10}s_5\dots s_2s_1$ in $\widetilde D_{11}$, see Figure \ref{F:CoxeterGraphDtilde}. Let $v={s_4s_5s_{10}s_{11}s_6s_7s_8s_9s_5s_4s_2s_1}$ and $u=s_8s_6s_1$ then $\overline v=s_1s_2s_4s_5s_{10}s_{11}s_6s_7s_8s_9s_5s_4$ and $\overline u= s_1s_6s_8$.

Now we give a graphical representation of the pair $(\overline v,\overline u)$. We start from the rooted tree of the Coxeter graph and we substitute for each vertex a table with one column and two rows. In the first row we write $\overline v(s_j)$ ($s_j$ is the element associated to the vertex); in the second row we write $\overline u(s_j)$. In the case $\overline v(s_j)=1$, it is possible that $s_j$ is on the left or on the right of $s_n$ (the root) as subword of $t$. We distinguish the two cases by writing $1_l$ if $s_j$ is on the left (as subword) of $s_n$, and $1_r$ otherwise. By definition we write $1_l$ in the root $s_n$ if $\overline v(s_n)\ne 0$. We apply the same rule to the second row. Moreover, in the first row, we use capital letter $R$ if the second row of the column to the right does not contain $0$.

We mark the column corresponding to $s_j$ with $\circ$ if $j\in J$ and with $\times$ if $j\not\in J$. Finally, if a vertex $s_j$ has only one vertex on the left then we write the two corresponding columns in a unique table. In Figure \ref{F:DiagD11} we give the graphical representation of the pair $(\overline v,\overline u)$ in $\widetilde D_{11}$, with $J=\{s_5,s_7\}$.

\begin{figure}[hbt]
$$
\xymatrix{
{\begin{tabular}{|c|}\multicolumn{1}{c}{$\times$}\\ \hline $1_l$\\ \hline $0$ \\ \hline\end{tabular}} \ar@{-}[rd] \\
{\begin{tabular}{|c|}\multicolumn{1}{c}{$\times$}\\ \hline $1_l$\\\hline $1_l$ \\\hline\end{tabular}} \ar@{-}[r]
&
{\begin{tabular}{|c|c|c|}\multicolumn{1}{c}{$\times$} &\multicolumn{1}{c}{$\times$} &\multicolumn{1}{c}{$\circ$} \\ \hline $0$ & $2 $ & $2$ \\\hline $0$ & $0$ & $0$ \\\hline\end{tabular}}\ar@{-}[r]
&
{\begin{tabular}{|c|}\multicolumn{1}{c}{$\times$}\\ \hline $1_l$\\\hline $1_l$ \\\hline\end{tabular}}\\
{\begin{tabular}{|c|}\multicolumn{1}{c}{$\times$}\\ \hline $1_l$\\\hline $0$ \\\hline\end{tabular}}\ar@{-}[r]
&
{\begin{tabular}{|c|c|c|}\multicolumn{1}{c}{$\times$} &\multicolumn{1}{c}{$\times$} &\multicolumn{1}{c}{$\circ$} \\ \hline $1_R$ & $1_r$ & $1_R$ \\\hline $0$ & $1_r$ & $0$ \\ \hline \end{tabular}}\ar@{-}[ru] \\
{\begin{tabular}{|c|}\multicolumn{1}{c}{$\times$}\\ \hline $1_l$\\\hline $0$ \\\hline\end{tabular}}\ar@{-}[ru]
}
$$
\caption{Diagram of $(\overline v=s_1s_2s_4s_5s_{11}s_{10}s_6s_7s_8s_9s_5s_4,\overline u s_1s_6s_8)$ in $\widetilde D_{11}$.}\label{F:DiagD11}
\end{figure}

In the sequel a symbol $*$ denotes the possibility to have arbitrary entries in the cell. A symbol such as $\not{1_l}$, $\not{0}$, etc.\ means that the value in the cell is not $1_l$, $0$, etc. Moreover we will be interested in subdiagrams of such representations, i.\ e.\ diagrams obtained by deleting one or more columns. Since  the order of the tables from  top to bottom  is not important (while  the order from  left to  right is fundamental), we use the following notation
$$
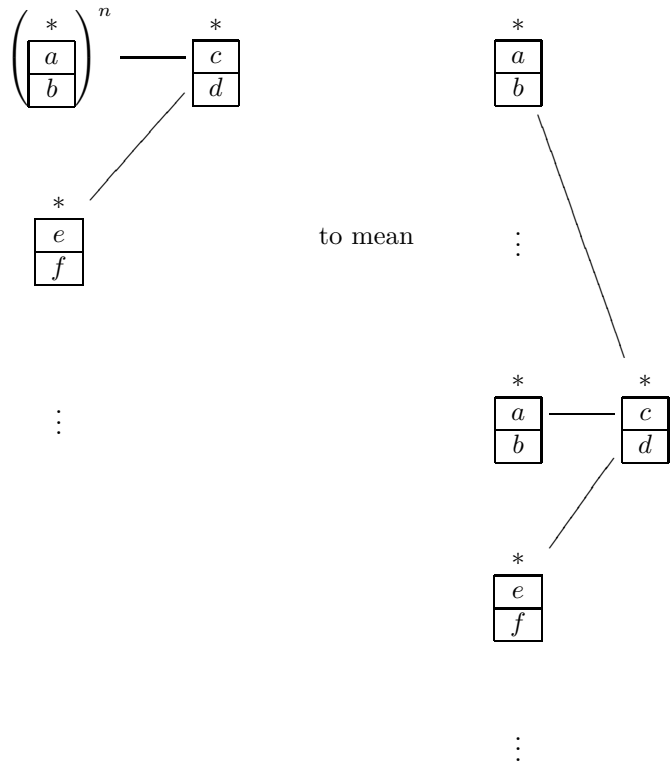
\xymatrix{
{\left(  \begin{tabular}{|c|}\multicolumn{1}{c}{$*$}\\ \hline $a$\\ \hline $b$ \\ \hline\end{tabular}\right)^n} \ar@{-}[r] &{\begin{tabular}{|c|}\multicolumn{1}{c}{$*$}\\ \hline $c$\\ \hline $d$\\ \hline\end{tabular}}
&  & {\begin{tabular}{|c|}\multicolumn{1}{c}{$*$}\\ \hline $a$\\ \hline $b$ \\ \hline\end{tabular}}\ar@{-}[ddr]\\
{\begin{tabular}{|c|}\multicolumn{1}{c}{$*$}\\ \hline $e$\\ \hline $f$ \\ \hline\end{tabular}}\ar@{-}[ur] & & {\text{to mean}} & {\vdots} \\
{\vdots} & &  &  {\begin{tabular}{|c|}\multicolumn{1}{c}{$*$}\\ \hline $a$\\ \hline $b$ \\ \hline\end{tabular}}\ar@{-}[r] & {\begin{tabular}{|c|}\multicolumn{1}{c}{$*$}\\ \hline $c$\\ \hline $d$\\ \hline\end{tabular}}\\
&  & &{\begin{tabular}{|c|}\multicolumn{1}{c}{$*$}\\ \hline $e$\\ \hline $f$ \\ \hline\end{tabular}}\ar@{-}[ur]\\
& & & {\vdots}
}
$$
where the column with entries $a,b$ is repeated $n$ times.
Now we give all the definitions necessary to  Theorem \ref{T:KLBoolean}.

Fix a diagram $d$. We use the notation $\# d'$, with $d'$ another diagram, to denote the number of subdiagrams of $d$ equal to $d'$.
Given a pair $(\overline v,\overline u)$ in $W$, we let
$$
a_h(\overline u,\overline v)=\#\left. %
\xymatrix{
{\left(\begin{tabular}{|c|}\multicolumn{1}{c}{$*$}\\ \hline $\not 2$\\ \hline $*$ \\ \hline\end{tabular}\right)^n}\ar@{-}[r] &
{\begin{tabular}{|c|}\multicolumn{1}{c}{$\times$}\\ \hline $1_*$\\ \hline $0$\\ \hline\end{tabular}}\\
{\left(\begin{tabular}{|c|}\multicolumn{1}{c}{$*$}\\ \hline $2$\\ \hline $\not 2$ \\ \hline\end{tabular}\right)^{h+1}}\ar@{-}[ur]
}\right. %
\quad+\quad\# \left.%
\xymatrix{
{\left(\begin{tabular}{|c|}\multicolumn{1}{c}{$*$}\\ \hline $\not 2$\\ \hline $*$ \\ \hline\end{tabular}\right)^n} \ar@{-}[r] &
{\begin{tabular}{|c|}\multicolumn{1}{c}{$\times$}\\ \hline $2$\\ \hline $0$\\ \hline\end{tabular}}\\
{\left(\begin{tabular}{|c|}\multicolumn{1}{c}{$*$}\\ \hline $2$\\ \hline $\not 2$ \\ \hline\end{tabular}\right)^{h}}\ar@{-}[ur]
}\right.
$$
$$
+\quad\#\xymatrix{{\left(  \begin{tabular}{|c|}\multicolumn{1}{c}{$*$}\\ \hline $\not 2$\\ \hline $*$ \\ \hline\end{tabular}\right)^n} \ar@{-}[rd]\\
{\begin{tabular}{|c|}\multicolumn{1}{c}{$*$}\\ \hline $1_l$\\ \hline $1_l$ \\ \hline\end{tabular}}\ar@{-}[r] &
{\begin{tabular}{|c|}\multicolumn{1}{c}{$\circ$}\\ \hline $1_*$\\ \hline $0$\\ \hline\end{tabular}}\\
{\left(\begin{tabular}{|c|}\multicolumn{1}{c}{$*$}\\ \hline $2$\\ \hline $\not 2$ \\ \hline\end{tabular}\right)^{h+1}}\ar@{-}[ur]}
+\quad\#\xymatrix{{\left(  \begin{tabular}{|c|}\multicolumn{1}{c}{$*$}\\ \hline $\not 2$\\ \hline $*$ \\ \hline\end{tabular}\right)^n} \ar@{-}[dr]\\
{\begin{tabular}{|c|}\multicolumn{1}{c}{$*$}\\ \hline $1_l$\\ \hline $1_l$ \\ \hline\end{tabular}}\ar@{-}[r] &
{\begin{tabular}{|c|}\multicolumn{1}{c}{$\circ$}\\ \hline $2$\\ \hline $0$\\ \hline\end{tabular}}\\
{\left(\begin{tabular}{|c|}\multicolumn{1}{c}{$*$}\\ \hline $2$\\ \hline $\not 2$ \\ \hline\end{tabular}\right)^{h}}\ar@{-}[ur]
}
+\quad\#\xymatrix{{\left(  \begin{tabular}{|c|}\multicolumn{1}{c}{$*$}\\ \hline $x$\\ \hline $y$ \\ \hline\end{tabular}\right)^n} \ar@{-}[r] &
{\begin{tabular}{|c|}\multicolumn{1}{c}{$\circ$}\\ \hline $1_{R}$\\ \hline $0$\\ \hline\end{tabular}}\\
{\left(\begin{tabular}{|c|}\multicolumn{1}{c}{$*$}\\ \hline $2$\\ \hline $\not 2$ \\ \hline\end{tabular}\right)^{h+1}}\ar@{-}[ur]
}
$$
$$
b_h(\overline u,\overline v)=\#\xymatrix{{\left(  \begin{tabular}{|c|}\multicolumn{1}{c}{$*$}\\ \hline $x$\\ \hline $y$ \\ \hline\end{tabular}\right)^n} \ar@{-}[r] &
{\begin{tabular}{|c|}\multicolumn{1}{c}{$\circ$}\\ \hline $1_{l/r}$\\ \hline $0$\\ \hline\end{tabular}}\\
{\left(\begin{tabular}{|c|}\multicolumn{1}{c}{$*$}\\ \hline $2$\\ \hline $\not 2$ \\ \hline\end{tabular}\right)^{h+1}}\ar@{-}[ur]
}
+\quad\#\xymatrix{{\left(  \begin{tabular}{|c|}\multicolumn{1}{c}{$*$}\\ \hline $x$\\ \hline $y$ \\ \hline\end{tabular}\right)^n} \ar@{-}[r] &
{\begin{tabular}{|c|}\multicolumn{1}{c}{$\circ$}\\ \hline $2$\\ \hline $0$\\ \hline\end{tabular}}\\
{\left(\begin{tabular}{|c|}\multicolumn{1}{c}{$*$}\\ \hline $2$\\ \hline $\not 2$ \\ \hline\end{tabular}\right)^{h}}\ar@{-}[ur]
}
$$
$$
c(\overline u,\overline v)=\#\xymatrix{{
\left(  \begin{tabular}{|c|}\multicolumn{1}{c}{$*$} \\ \hline $x$\\ \hline $y$ \\ \hline\end{tabular}\right)^n} \ar@{-}[r] &
{\begin{tabular}{|c|}\multicolumn{1}{c}{$\circ$}\\ \hline $2$\\ \hline $0$\\ \hline\end{tabular}}}
+\quad\#\xymatrix{{\left(  \begin{tabular}{|c|}\multicolumn{1}{c}{$*$}\\ \hline $x$\\ \hline $y$ \\ \hline\end{tabular}\right)^n}\ar@{-}[r] &
{\begin{tabular}{|c|}\multicolumn{1}{c}{$\circ$}\\ \hline $1_l$\\ \hline $0$\\ \hline\end{tabular}}\\
{\begin{tabular}{|c|}\multicolumn{1}{c}{$*$}\\ \hline $*$\\ \hline $0$ \\ \hline\end{tabular}}\ar@{-}[ur]
}
+\quad\#\xymatrix{{\left(  \begin{tabular}{|c|}\multicolumn{1}{c}{$*$}\\ \hline $x$\\ \hline $y$ \\ \hline\end{tabular}\right)^n}\ar@{-}[r] &
{\begin{tabular}{|c|}\multicolumn{1}{c}{$\circ$}\\ \hline $1_r$\\ \hline $0$\\ \hline\end{tabular}}\\
{\begin{tabular}{|c|}\multicolumn{1}{c}{$*$}\\ \hline $2$\\ \hline $\not 2$ \\ \hline\end{tabular}}\ar@{-}[ur]
}
$$
$$
+\quad\#\xymatrix{{\left(  \begin{tabular}{|c|}\multicolumn{1}{c}{$*$}\\ \hline $x$\\ \hline $y$ \\ \hline\end{tabular}\right)^n}\ar@{-}[r] &
{\begin{tabular}{|c|}\multicolumn{1}{c}{$\circ$}\\ \hline $1_r$\\ \hline $0$\\ \hline\end{tabular}}
}
+\quad\#\xymatrix{{\left(  \begin{tabular}{|c|}\multicolumn{1}{c}{$*$}\\ \hline $x$\\ \hline $y$ \\ \hline\end{tabular}\right)^n} \ar@{-}[r] &
{\begin{tabular}{|c|}\multicolumn{1}{c}{$\circ$}\\ \hline $1_l$\\ \hline $0$\\ \hline\end{tabular}}\\
{\begin{tabular}{|c|}\multicolumn{1}{c}{$*$}\\ \hline $2$\\ \hline $\not 2$ \\ \hline\end{tabular}}\ar@{-}[ur]
}
$$
$$
c'(\overline u,\overline v)=
\#\quad\xymatrix{{\left(  \begin{tabular}{|c|}\multicolumn{1}{c}{$*$}\\ \hline $x'$\\ \hline $y'$ \\ \hline\end{tabular}\right)^n }\ar@{-}[r] &
{\begin{tabular}{|c|}\multicolumn{1}{c}{$\circ$}\\ \hline $2$\\ \hline $1_r$\\ \hline\end{tabular}}
}
+\quad\#\xymatrix{{\left(  \begin{tabular}{|c|}\multicolumn{1}{c}{$*$}\\ \hline $x'$\\ \hline $y'$ \\ \hline\end{tabular}\right)^n} \ar@{-}[r] &
{\begin{tabular}{|c|}\multicolumn{1}{c}{$\circ$}\\ \hline $1_l$\\ \hline $1_l$\\ \hline\end{tabular}}\\
{\begin{tabular}{|c|}\multicolumn{1}{c}{$*$}\\ \hline $2$\\ \hline $\not 2$ \\ \hline\end{tabular}}\ar@{-}[ur]
}
$$
$$c''(\overline u,\overline v)=
\#\quad\xymatrix{{\left(  \begin{tabular}{|c|}\multicolumn{1}{c}{$*$}\\ \hline $x'$\\ \hline $y'$ \\ \hline\end{tabular}\right)^n}\ar@{-}[r] &
{\begin{tabular}{|c|}\multicolumn{1}{c}{$\circ$}\\ \hline $*$\\ \hline $1_l$\\ \hline\end{tabular}}\\
{\begin{tabular}{|c|}\multicolumn{1}{c}{$*$}\\ \hline $2$\\ \hline $1_l$ \\ \hline\end{tabular}}\ar@{-}^{m(s,s')=3}[ur],
}
$$
where $(x,y)\in P_1$, $(x',y')\in P_1\cup P_2$ with $P_1=\{(1_l,0),(1_r,0),(1_r,1_r),(2,1_r)\}$, $P_2=\{(1_R,0),(1_R,1_r),(2,0)\}$. In each diagram  $(x,y)$, $(x',y')$, $(\not 2,*)$ or $( 2,\not 2)$ are not necessarily the same pair for all $n\ge 0$ (or $h\ge 0$) columns.   We can now state the main result of this work.
\begin{theorem}
\label{T:KLBoolean}
Let $J\subseteq S$, $u,v\in W^J$ and set $\overline c(\overline u,\overline v)=c(\overline u,\overline v)+c'(\overline u,\overline v)+c''(\overline u,\overline v)$. We have
$$
P^{J}_{u,v}(q)=\left\{\begin{array}{ll} \prod_{h\ge 1} f_{h+1}^{a_h}(f_{h+1}-1)^{b_h} & \text {if $\overline c(\overline u,\overline v)=0$}\\ 0 & \text{otherwise}\end{array} \right.
$$
\end{theorem}

\begin{corollary}
\label{C:TopCoeff}
Let $J\subseteq S$, $u,v\in W^J$ with $l(v)-l(u)\ge 3$ odd. Then $\mu (u,v)\ne 0$ if and only if the entries in each column of the diagram of $(\overline u,\overline v)$ are equal except for exactly one subdiagram which is 
$$
\xymatrix{{
\left(  \begin{tabular}{|c|}\multicolumn{1}{c}{$*$} \\ \hline $2$\\ \hline $1_l$ \\ \hline\end{tabular}\right)^{h+1}} \ar@{-}[r] &
{\begin{tabular}{|c|}\multicolumn{1}{c}{$*$}\\ \hline $\not 0$\\ \hline $0$\\ \hline\end{tabular}}}
\text{ or }
\xymatrix{{
\left(  \begin{tabular}{|c|}\multicolumn{1}{c}{$*$}\\ \hline $2$ \\ \hline $1_l$ \\ \hline\end{tabular}\right)^h} \ar@{-}[r] &
{\begin{tabular}{|c|c|c|}\multicolumn{1}{c}{$*$}& \multicolumn{1}{c}{}&\multicolumn{1}{c}{$*$}\\ \hline $2$& $\dots$ & $2$\\ \hline $0$ & $\dots$ & $0$\\ \hline\end{tabular}}}
$$
In this case $\mu(u,v)=C([\frac{h+1}{2}])$, the $[\frac{h+1}{2}]$-th Catalan number.
\end{corollary}

\begin{proof}
If in the diagram of $(\overline u,\overline v)$ there are more that one subdiagram with the properties in the statement, then by Theorem \ref{T:KLBoolean} $P_{u,v}^{J}$ is the product of at least two factors. Then the degree of $P_{u,v}^{J}$ is at most $l(v)-l(u)-2$. The last part of the statement follows by properties of $f_h(q)$.
\end{proof}

We now  prove Theorem \ref{T:KLBoolean}.
\begin{proof}
We argue by induction on $l(v)$. If $l(v)=1$ then $P_{u,v}^{J}=1$, since $u\le v$ and the result is trivial. Now suppose $l(v)\ge 2$. Let $C$ be one of the leftmost columns in the diagram. The entries of $C$ can be filled by several values.

We first consider the case that $C$ contains the pair $(1_R,0)$. Let $s\in S$ be the element corresponding to $C$. Then $s\in D_R(v)$ and $us\not\le vs$, $us\in W^J$ (since R is a capital letter). Therefore $P_{u,v}^{J}=P_{u,vs}^{J}$, i.\ e.\ we can remove the column $C$ from the diagram without changing the Kazhdan--Lusztig polynomial. The statement follows by induction.

Let's suppose that $C$ contains the pair $(1_r,0)$. As before $s\in D_R(v)$. If $s\in J$ then $us\not\in W^J$, therefore $P_{u,v}^{J}=0$ and this agrees with the fourth summand in $c(\overline u,\overline v)$. Otherwise, $us\in W^J$ and $us\not\le vs$, therefore $P_{u,v}^{J}=P_{u,vs}^{J}$ and the claim follows by induction.

If $C$ contains the pair $(1_R,1_r)$ or $(1_r,1_r)$ then $u\not\le vs$ and therefore $P_{u,v}^{J}=P_{us,vs}^{J}$. The statement follows.

Now suppose that $C$ contains $(1_l,1_l)$. By Lemma \ref{L:11} $P_{u,v}^{J}=P_{su,sv}^{J\cap \Com(s)}$. This is equivalent to deleting the column $C$ and  putting an $\times$ above the column on the right of $C$. This agrees with the assumption $(1_l,1_l)\notin P_1\cup P_2$. If $C$ contains $(2,2)$ then $P_{u,v}^{J}=P_{us,vs}^{J}$ and we are in the case $(1_l,1_l)$. Again this agrees with the assumption $(2,2)\notin P_1\cup P_2$.

If $C$ contains $(1_l,0)$ then by Lemma \ref{L:10} $P_{u,v}^{J}=P_{u,sv}^{J}$ except in the case $sv\notin W^J$. Then we have to exclude 
$$
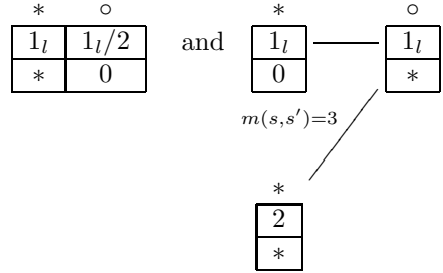
\begin{tabular}{|c|c|}\multicolumn{1}{c}{$*$} & \multicolumn{1}{c}{$\circ$}\\ \hline $1_l$ & $1_l/2$\\ \hline $*$ & $0$\\ \hline\end{tabular}\quad \text{ and }
\xymatrix{{ \begin{tabular}{|c|}\multicolumn{1}{c}{$*$}\\ \hline $1_l$\\ \hline $0$ \\ \hline\end{tabular}} \ar@{-}[r] &
{\begin{tabular}{|c|}\multicolumn{1}{c}{$\circ$}\\ \hline $1_l$\\ \hline $*$\\ \hline\end{tabular}}\\
{\begin{tabular}{|c|}\multicolumn{1}{c}{$*$}\\ \hline $2$\\ \hline $*$ \\ \hline\end{tabular}}\ar@{-}^{m(s,s')=3}[ur]
}
$$
These diagrams are included in $c''(\overline u,\overline v)$, in the first summand of $c'(\overline u,\overline v)$ and in three summands of $c(\overline u,\overline v)$.
If $C$ contains $(2,1_r)$ then $P_{u,v}^{J}=P_{us,vs}^{J}$ and we are in the case $(1_l,0)$.

Now suppose that $C$ contains $(2,0)$ and the second entry in the column on the right is non-zero. Then $P_{u,v}^{J}=P_{u,vs}^{J}$ since $us\not \le vs$ and there is no $w\in W^J$ with $u\le w<vs$ and $ws<w$\footnote{$ws<w\le vs$ implies that $s$ is an occurrence in the first place of the word $\overline w$  (since the same is for the word $\overline{vs}$); therefore $s\in D_L(w)\cap D_R(w)$ and so if we denote with $t$ the element on the right of $s$ then $w(t)=0$ but this implies that $u\not \le w$.}. Then we are in the case $(1_l,0)$ and this agrees with the assumption $(2,0)\in P_2$.

If $C$ contains $(2,1_l)$ and the second entry in the column on the right is non-zero, then $us\notin W^J$ if and only if the diagram is such as in $c''(\overline u,\overline v)$. Otherwise $P_{u,v}^{J}=P_{u,vs}^{J} $ since, as before, there is no $w\in W^J$ with $u\le w<vs$ and $ws<w$. Then we are in the case $(1_l,1_l)$.

Finally we have to consider the cases $(2,1_l)$ or $(2,0)$ with the second entry in the column on the right equal to $0$. By Proposition \ref{P:PropertiesKL}, they can be treated as a unique case. Therefore we assume that $C$ contains $(2,1_l)$.

For the diagram
\begin{equation}
\label{E:Case1}
\xymatrix{{
\left(  \begin{tabular}{|c|}\multicolumn{1}{c}{$*$} \\ \hline $2$\\ \hline $1_l$ \\ \hline\end{tabular}\right)^n} \ar@{-}[r] &
{\begin{tabular}{|c|}\multicolumn{1}{c}{$*$}\\ \hline $1_*$\\ \hline $0$\\ \hline\end{tabular}}}
\end{equation}
the corresponding Kazhdan--Lusztig polynomial is  $P_{u,v}^{J}=f_n-\alpha$, where $\alpha=1$ when there are $\circ$ and $1_l$ on the rightmost column and $\alpha=0$ otherwise. In fact, by induction and  by Proposition \ref{P:KLFromDesc}, $P_{u,v}^{J} = f_{n-1}(q)-\alpha + q f_{n-1}(q)- \mu(f_{n-1}(q))q^\frac{n-1}{2}$.  Note that by induction, if $\mu(w,vs)\ne 0$ then necessarily the diagram of $w$ coincides with the diagram of $v$ in all other  columns not depicted in (\ref{E:Case1}). By Lemma \ref{L:Propf} we get $P_{u,v}^{J}=f_n-\alpha$ (note that if $n=1$, $f_1-1=0$ and this agrees with the $3^{rd}$ and $5^{th}$ summands in $c(\overline u,\overline v)$).

For the last subcase,
\begin{equation}
\label{E:Case2}
\#\xymatrix{{
\left(  \begin{tabular}{|c|}\multicolumn{1}{c}{$*$} \\ \hline $2$\\ \hline $1_l$ \\ \hline\end{tabular}\right)^n}\ar@{-}[r] &
{\begin{tabular}{|c|}\multicolumn{1}{c}{$*$}\\ \hline $2$\\ \hline $0$\\ \hline\end{tabular}}}
\end{equation}
the analysis is a bit harder. Let's assume that on the right of this diagram there is a sequence of $m$ columns
\begin{equation}
\label{E:Case2bis}
 \begin{tabular}{|c|c|c|c|}\multicolumn{1}{c}{$*$}& \multicolumn{1}{c}{$*$}& \multicolumn{1}{c}{} & \multicolumn{1}{c}{$*$}\\ \hline $2$ & $2$ & $\cdots$ & $ 2$ \\ \hline $0$ & $0$ & $\cdots$ & $0$ \\ \hline\end{tabular}
\end{equation}
ending with a column whose entries are not $(2,0)$ or with a column corresponding to a vertex of degree greater than $2$. Suppose that exactly $k$  of these columns have a $\circ$ and the other $m-k$ have a $\times$. 
Let $\overline P_{u,v}^{J}$ be the Kazhdan--Lusztig polynomian corresponding to the diagram of $(\overline u, \overline v)$ after deleting the subdiagrams depicted in (\ref{E:Case2}) and (\ref{E:Case2bis}). By appling Proposition \ref{P:KLFromDesc}  we get by induction
$$
P_{u,v}^{J}=(f_{n}(q)-\alpha)q^k(1+q)^{m-k}\overline P_{u,v}^{J}+f_n(q) q^{k+1}(1+q)^{m-k}\overline P_{u,v}^{J}-\widetilde M_{u,v}
$$
where $\alpha=1$ if there is a $\circ$  on the rightmost column of (\ref{E:Case2}) and $\alpha=0$ otherwise, and $\widetilde M_{u,v}$ is the sum in Proposition \ref{P:KLFromDesc}. Note that by induction and Corollary \ref{C:TopCoeff} $\mu(w,vs)\ne 0$ only if  the diagram of $w$ coincides with that of $v$ in all the columns not depicted in (\ref{E:Case2}) and (\ref{E:Case2bis}).
More precisely, for any such $w$, the diagram of $(\overline w,\overline {vs})$ is of the form
$$
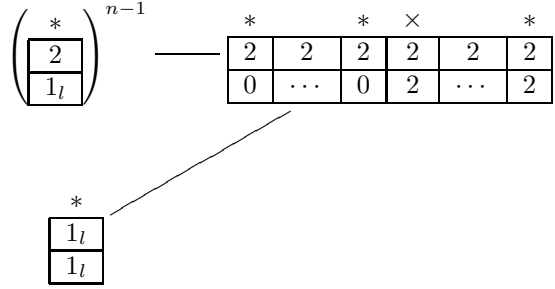
\xymatrix{{
\left(  \begin{tabular}{|c|}\multicolumn{1}{c}{$*$} \\ \hline $2$\\ \hline $1_l$ \\ \hline\end{tabular}\right)^{n-1}} \ar@{-}[r] &
{\begin{tabular}{|c|c|c|c|c|c|}\multicolumn{1}{c}{$*$} &\multicolumn{1}{c}{} &\multicolumn{1}{c}{$*$} &\multicolumn{1}{c}{$\times$} &\multicolumn{1}{c}{} &\multicolumn{1}{c}{$*$} \\ \hline $ 2$&$ 2$&$ 2$&$ 2$&$ 2$&$ 2$\\ \hline $0$ & $\cdots$ & $0$ & $2$ & $\cdots$ & $2$ \\ \hline\end{tabular}}\\
{\begin{tabular}{|c|}\multicolumn{1}{c}{$*$} \\ \hline $1_l$\\ \hline $1_l$ \\ \hline\end{tabular}}\ar@{-}[ur]
}
$$
and in all other columns the first entries are equal to the second entries.
Therefore $\widetilde M_{u,v}$ is
$$
\overline P^{J}_{u,v}\mu(f_n(q))\big(q^{\frac{n-2}{2}+k}(q+1)^{m-k-1}+q^{\frac{n-2}{2}+k+1}(q+1)^{m-k-2}+\cdots + q^{\frac{n-2}{2}+m-1}+q^{\frac{n-2}{2}+m}  \big)
$$
if $n$ is even and $0$ if $n$ is odd.
In this formula the powers of $q$ include both the contributions of $q^{\frac{l(w,vs)}{2}}$ and of $P_{u,w}^{J}$. In the case $n$ even, we have 
\begin{align*}
\widetilde M_{u,v}=&
\overline P^{J}_{u,v} \mu(f_n(q)) \big( q^{\frac{n-2}{2}+k}((q+1)^{m-k}-q^{m-k})+q^{\frac{n-2}{2}+m}\big)\\
&=\overline P^{J}_{u,v} \mu(f_n(q))(q+1)^{m-k}q^{\frac{n-2}{2}+k}
\end{align*}
and therefore 
$$
P_{u,v}^{J}=\overline P^{J}_{u,v}q^k(1+q)^{m-k}\big(f_n(q)-\alpha+qf_n(q)-\mu({f_n(q)})q^{\frac{n-2}{2}}\big)=P^{J}_{u,v}q^k(1+q)^{m-k}(f_{n+1}(q)-\alpha)
$$
by Lemma \ref{L:Propf}. Analogously, if $n$ is odd we have
$$
P_{u,v}^{J}=\overline P_{u,v}q^k(1+q)^{m-k}\big(f_n(q)-\alpha+qf_n(q)\big)=\overline P_{u,v}q^k(1+q)^{m-k}(f_{n+1}(q)-\alpha).
$$
The cases $n=1$ and $n=2$ are similar (note that $f_1(q)-\alpha=0$ if $\alpha=1$). Thus the proof is completed.
\end{proof}

In the case of the classical Kazhdan--Lusztig polynomias, Theorem \ref{T:KLBoolean} becomes much simpler.

\begin{corollary}
Let $W$ be a tree-Coxeter group and $u,v\in W$ be boolean elements. Then $P_{u,v}(q)=\prod_{h\ge 1}f_{h+1}^{a_h}$, where $a_h$ is defined before Theorem \ref{T:KLBoolean}.
\end{corollary}

For example, the Kazhdan--Lusztig polynomial of the pair $(u,v)$ depicted in Figure \ref{F:DiagD11} is $P_{u,v}^{J}=f_2(q)-1=q$, since $a_h=0$ for all $h\ge 0$, $b_1=1$ and $b_h=0$ for all $h\ne 1$.

\begin{remark}
Theorem \ref{T:KLBoolean} implies result in \cite[Theorem 5.2]{Marietti2010}.
\end{remark}
We give the following easy consequence of Theorem \ref{T:KLBoolean} which proves, in the case of boolean elements, a conjecture of Brenti \cite{BrentiPr}.
\begin{corollary}
\label{C:conjB}
Let $I\subseteq J$ and $u,v\in W^J$. Then 
$$
P_{u,v}^{J}(q)\le P_{u,v}^{I}(q)
$$
in the coefficientwise comparison.
\end{corollary}

Now we consider the case of $\widetilde A_n$ for $n\ge 2$ ($\widetilde A_1$ is a tree-Coxeter group).
The Coxeter diagram of $\widetilde A_n$ is a cycle, therefore we can not apply Theorem \ref{T:KLBoolean}. However we use the same arguments of its proof to have an analogue result. Consider a boolean reflection $t$ in $\widetilde A_n$ of length $2n+1$. Then $t=s_{i+1}s_{i+2}\cdots s_{n}s_0\cdots s_{i-1}s_is_{i-1}\cdots s_0s_n\cdots s_{i+2}s_{i+1}$ for some $i\in[0,n]$ (the indices are modulo $n+1$): for any reduced boolean word of $t$ fix the first and the last letter, then change all other letters as a subword in $A_n$. For any pair $(u,v)\in W^2$, $u\le v\le t$ we depict a diagram whose rightmost column contains $(\overline u(s_i), \overline v(s_i))$. The leftmost column contains $(\overline u(s_{i+1}), \overline v(s_{i+1}))$ and the other columns are defined by following the cyclic Coxeter diagram of $\widetilde A_n$. See Figure \ref{F:Atilde}  for an example.

\begin{figure}
$$
\xymatrix{{
\begin{tabular}{|c|}\multicolumn{1}{c}{$\times$}\\
            \hline $2$\\ \hline $1_l$
             \\ \hline\end{tabular} }
\ar@{-}[rr] &
&{\begin{tabular}{|c|}\multicolumn{1}{c}{$\times$}\\ \hline $1_l$\\ \hline $1_l$ \\ \hline\end{tabular}}\\
 &
{\begin{tabular}{|c|c|c|}\multicolumn{1}{c}{$\times$} &\multicolumn{1}{c}{$\times$} &\multicolumn{1}{c}{$\circ$} \\ \hline $0$ & $2$ & $2$ \\\hline $0$ & $0$ & $0$ \\\hline\end{tabular}}\ar@{-}[ul]\ar@{-}[ur]
}
$$
\caption{Diagram of $(\overline v =s_0s_2s_3s_4s_3s_2s_0\overline u=s_0s_4)$ in $\widetilde A_4$, with boolean reflection $t=s_0s_1s_2s_3s_4s_3s_2s_1s_0$ and $J=\{s_3\}$.}
\label{F:Atilde}
\end{figure}

In the follows we assume that $\overline v(s_i)=1_l$ where $i$ is the central letter in the word of the boolean reflection $t$ and $\overline v(s_j)\ne 0$ where $s_j$ corresponds to the second column on the left. In fact, if $\overline v(s_i)=0$ then $v$ can be identified as an element in  $A_n$ and we can apply Theorem \ref{T:KLBoolean}.

We define
$$
a(\overline u, \overline v)=
\# \begin{tabular}{|c|c|}\multicolumn{1}{c}{$*$}&\multicolumn{1}{c}{$\times$}\\\hline $2$ & $2$\\\hline $*$ & $0$\\\hline\end{tabular}
+\quad\#\xymatrix{{
\begin{tabular}{|c|}\multicolumn{1}{c}{$\times$}\\
            \hline $2$\\ \hline $*$
             \\ \hline\end{tabular} }\ar@{-}[rr]
&
&{\begin{tabular}{|c|}\multicolumn{1}{c}{$\times$}\\ \hline $1_l$\\ \hline $0$ \\ \hline\end{tabular}}\\
 &
{\begin{tabular}{|c|c|}\multicolumn{1}{c}{$\times$} &\multicolumn{1}{c}{}  \\ \hline $\not 0$ & $\cdots$ \\\hline $0$ & $\cdots$  \\\hline\end{tabular}}\ar@{-}[ul]\ar@{-}[ur]
}
$$

$$
b(\overline u, \overline v)=
\# \begin{tabular}{|c|c|}\multicolumn{1}{c}{$*$}&\multicolumn{1}{c}{$\circ$}\\\hline $2$ & $2$\\\hline $*$ & $0$\\\hline\end{tabular}
+\quad\#\xymatrix{{
\begin{tabular}{|c|}\multicolumn{1}{c}{$\times$}\\
            \hline $2$\\ \hline $*$
             \\ \hline\end{tabular}}\ar@{-}[rr] 
&
&{\begin{tabular}{|c|}\multicolumn{1}{c}{$\circ$}\\ \hline $1_l$\\ \hline $0$ \\ \hline\end{tabular}}\\
 &
{\begin{tabular}{|c|c|}\multicolumn{1}{c}{$*$} &\multicolumn{1}{c}{}  \\ \hline $\not 0$ & $\cdots$ \\\hline $0$ & $\cdots$  \\\hline\end{tabular}}\ar@{-}[ul]\ar@{-}[ur]
}
+\quad\#\xymatrix{{
\begin{tabular}{|c|}\multicolumn{1}{c}{$\times$}\\
            \hline $2$\\ \hline $*$
             \\ \hline\end{tabular} }	\ar@{-}[rr]
&
&{\begin{tabular}{|c|}\multicolumn{1}{c}{$\times$}\\ \hline $1_l$\\ \hline $0$ \\ \hline\end{tabular}}\\
 &
{\begin{tabular}{|c|c|}\multicolumn{1}{c}{$\circ$} &\multicolumn{1}{c}{}  \\ \hline $\not 0$ & $\cdots$ \\\hline $0$ & $\cdots$  \\\hline\end{tabular}}\ar@{-}[ul]\ar@{-}[ur]
}
$$


$$
c'''(\overline u, \overline v)=
\#\xymatrix{{
\begin{tabular}{|c|}\multicolumn{1}{c}{$\times$}\\
            \hline $x'$\\ \hline $y'$
             \\ \hline\end{tabular}}\ar@{-}[rr]
&
&{\begin{tabular}{|c|}\multicolumn{1}{c}{$*$}\\ \hline $1_l$\\ \hline $*$ \\ \hline\end{tabular}}\\
 &
{\begin{tabular}{|c|c|} \multicolumn{1}{c}{$\circ$} &\multicolumn{1}{c}{}  \\ \hline $\not 0$ &$\cdots$ \\\hline  $*$ &$\cdots$  \\\hline\end{tabular}}\ar@{-}[ul]\ar@{-}[ur]
}
+\quad\#\xymatrix{{
\begin{tabular}{|c|}\multicolumn{1}{c}{$\times$}\\
            \hline $x'$\\ \hline $y'$
             \\ \hline\end{tabular} }\ar@{-}[rr]
&
&{\begin{tabular}{|c|} \multicolumn{1}{c}{$\circ$}\\ \hline  $1_l$\\ \hline $0$ \\ \hline\end{tabular}}\\
 &
{\begin{tabular}{|c|c|} \multicolumn{1}{c}{$*$} &\multicolumn{1}{c}{}  \\ \hline $\not 0$ &$\cdots$ \\\hline  $0$ & $\cdots$  \\\hline\end{tabular}}\ar@{-}[ul]\ar@{-}[ur]
}
$$
$$
+\quad\#\xymatrix{{
\begin{tabular}{|c|}\multicolumn{1}{c}{$\times$}\\
            \hline $x'$\\ \hline $y'$
             \\ \hline\end{tabular} }	\ar@{-}[rr]
&
&{\begin{tabular}{|c|}\multicolumn{1}{c}{$\circ$}\\ \hline $1_l$\\ \hline $*$ \\ \hline\end{tabular}}\\
 &
{\begin{tabular}{|c|c|} \multicolumn{1}{c}{$*$} &\multicolumn{1}{c}{}  \\ \hline $2$ & $\cdots$ \\\hline  $*$ &$\cdots$  \\\hline\end{tabular}}\ar@{-}[ul]\ar@{-}[ur]
}
$$
where $(x,y),(x',y')\in\{(1_l,0),(1_r,0),(1_r,1_r),(2,1_r)\}$. Moreover $(x',y')$ could be $(2,0)$ (respectively $(2,1_l)$)  if there is a non-zero entry (resp.\ exactly one non-zero entry  with a $\circ$) in the second row of one of the two columns on the right of the first column.

\begin{theorem}
\label{T:Atilde}
Let $u,v\in \widetilde A_n$ boolean reflection. Then
$$
P_{u,v}^{J}=\left\{ \begin{array}{ll} q^{b(\overline u,\overline v)}(1+q)^{a(\overline u,\overline v)}& \text{if $c(\overline u,\overline v)+c'(\overline u,\overline v)+c'''(\overline u,\overline v)=0$}\\
0 &\text{otherwise}\end{array}\right.
$$
\end{theorem}

The proof is the same as in Theorem \ref{T:KLBoolean}. Delete the leftmost column if it contains $(1_*,0),(1_*,1_*),(2,2)$ by using Lemmas \ref{L:10} and \ref{L:11}. If it contains the pair $(2,\not 2)$ then consider the case with the second entries of both column on the right be zero and non-zero. In the first case apply Proposition \ref{P:KLFromDesc} and note that $\widetilde M_{u,v}=0$. We left to the reader all details.

\begin{remark}
For the classical Kazhdan--Lusztig polynomials, Theorem \ref{T:Atilde} reduces to \cite[Theorem 4.4]{Marietti2006}.
\end{remark}

\section {Kazhdan--Lusztig polynomials of boolean signed permutation}
\label{S:CombKL}
In this  section we consider the combinatorial interpretation of the finite Coxeter groups $A_n$, $B_n$ and $D_n$ as (signed) permutations and restate Theorem \ref{T:KLBoolean} by using statistics of such permutations.
We recall (see e.\ g.\ \cite[Chap. 1, 8]{Bjorner2005}) that $A_n$ is the group of permutations of the set $\{1,\dots,n+1\}$, $B_n$ is the set of permutations $\pi$ of $\{-n,-n+1,\dots,-1,1,\dots,n-1,n\}$ such that $\pi(-i)=-\pi(i)$ for all $i\le n$, and $D_n $ is the subset of permutations $\pi\in B_n$ such that the cardinality $\# (\{\pi(1),\dots, \pi(n)\}\cap \{-1,\dots,-n\})$ is even.
Note that each permutation $\pi$ of $A_n$, $B_n$ and $D_n$ is uniquely determined by $[\pi(1),\dots, \pi(n)]$. We call this sequence the \emph{window notation} of $\pi$.

Given a (signed) permutation $\pi$, if $\pi(i)>i$ we say that $\pi(i)$ is a \emph{top excedance} and $i$ is a \emph{bottom excedance} of $\pi$ .

It is well known that the set of all reflections in $A_n$ is given by transpositions $(i,j)$, with $i,j\le n+1$. Any such transposition admits $s_is_{i+1}\cdots s_{j-2}s_{j-1}s_{j-2}\cdots s_{i+1}s_i$ as reduced expression. So every reflection in the symmetric group is boolean and an element $\pi$ is boolean if and only if it is smaller than the top transposition $(1,n+1)$, i.\ e.\ $\pi$ admits a reduced expression which is a subword of $s_1\cdots s_{n-1}s_ns_{n-1}\cdots s_1$.

\begin{lemma}
\label{L:TranAn}
Let $\pi\in A_n$. Then $\pi$ is a boolean element if and only if $\#( \pi(\{1,\dots, i\})\cap \{1,\dots,i\})\ge i-1$ for all $i\le n$. 

Moreover, if $\overline \pi$ is the reduced expression of $\pi$, subword of $s_1\cdots s_{n}\cdots s_1$, then $\overline\pi(s_i)=1_l$ if $i+1$ is a top excedance of $\pi$; $\overline\pi(s_i)=1_r$ if $i+1$ is a top excedance of $\pi^{-1}$; $\overline\pi(s_i)=2$ if and only if $\pi(i+1)=i+1$ and $\pi(\{1,\dots, i\})\ne \{1,\dots,i\}$; $\overline\pi(s_i)=0$ if and only if $\pi(\{1\,\dots,i\})=\{1\,\dots,i\}$. 
\end{lemma}
\begin{proof}
We prove the first part of the statement by induction on $n$. If $n=1$ there is nothing to prove. Suppose that $n\ge 2$. Now $\pi$ is the product of $s_1$ (on the left, right or both) with a boolean element generated by $s_2,\dots,s_n$. Multiplying by $s_1$ on the right is equivalent to exchanging the first and the second elements in the windows notation of $\pi$; multiplying by $s_1$ on the left is equivalent to exchanging the elements $1$ and $2$ in the window notation of $\pi$.
It is easy to see that $\# (\pi(\{1,\dots, i\})\cap \{1,\dots,i\})$ does not change for $i\ge 2$ and that the claim is always true for $i=1$. The result follows by induction.

Vice versa, let $\pi\in A_n$ as in the statement. If $\pi(1)=1$ then $\pi$ can be identified with a permutation in $A_{n-1}$ (in the following we will say that $\pi\in A_{n-1}$) and the claim is true by induction. Now suppose that $\pi(1)\ne 1$. If $\pi(2)=1$ then $\pi s_1\in A_{n-1}$ so we can apply induction. If $\pi(1)=2$ then $s_i\pi\in A_{n-1}$ and  we can apply induction again. Then we have to consider the case $\pi(1)\ne 1,2$ and $\pi(2)\ne 1$.  Since $\# (\pi(\{1,2\})\cap\{1,2\}) \ge 1$ it forces to $\pi(2)=2$ and then $s_1\pi s_1\in A_{n-1}$. The claim follows.

Now we prove the second part. Fix an index $i\le n$. Let $\overline \pi'$ be the subword of $\overline \pi$ with only letters $s_{i+1},\dots, s_n$. Then $s_i\overline\pi'(i)=i+1$. If we multiply $s_i\overline\pi'$ by $s_j$, $j<i$, on the left or on the right, then the element $i+1$ may be moved on the left in the windows notation. Therefore $i+1$ is a top excedance of $\pi$ if $\overline\pi(s_i)=1_l$. Analogously, $\overline\pi's_i(i+1)=i$ and if we mutiply $s_i\overline\pi'$ by $s_j$, $j<i$, on the left or on the right, the element in $i+1$-th place in the window notation may be replaced with an element smaller that $i$. Therefore $i+1$ is a top excedance of $\pi^{-1}$. The third case is similar since $s_i\overline\pi' s_i(i+1)=i+1$. The last case is trivial. \end{proof}

Given $\pi\in A_n$, we define the following sets.
\begin{align*}
\Exc(\pi)=&\{i\in [n]\vert i+1 \text{ is a top excedance for $\pi$} \};\\
\Fix(\pi)=&\{i\in [n]\vert \pi([i])=[i]\};\\
\NFix(\pi)=&\{i\in [n]\setminus \Fix(\pi)\vert \pi(i+1)=i+1\}.
\end{align*}

Then by Theorem \ref{T:KLBoolean} and Lemma \ref{L:TranAn} we have
\begin{corollary}
\label{C:An}
Let $\pi,\rho\in A_n^J$ be two boolean permutations of $[n+1]$ such that $\pi\le \rho$ in the Bruhat order. Then the Kazhdan--Lusztig polynomial $P_{\pi,\rho}^{J}$ is zero if and only if there exists an index $i\le n$ such that one of the following condition is satisfied (we identify each $s_i\in J$ with $i$)
\begin{itemize}
	\item $i\in \Exc(\rho)\cap \Fix(\pi)$ and $i+1\in J\cap \NFix(\rho)$;
	\item $i,i+1\in \Exc(\rho)\cap \Fix(\pi)$ and $i+1\in J$;
	\item $i\in\Exc(\rho^{-1})\cap J$, $i,i+1\in \Fix(\pi)$,  and  $i-1\notin \Exc(\pi)\cap \Exc(\rho)$;
	\item $i,i+1\in \NFix(\rho)\cap \Exc(\pi^{-1})$ or $i,i+1\in \NFix(\rho)\cap \Exc(\pi)$ and $i+1\in J$;
	\item $i,i+1\in \NFix(\rho)$, $\#(\{i,i+1\}\cap \Exc(\pi^{-1}))=1$, $\#(\{i,i+1\}\cap \Fix(\pi))=1$ and $i+1\in J$.
\end{itemize}
In all other cases, let 
\begin{align}
\label{E:A}
A_{\pi,\rho}=\{i\in [n]\vert i,i+1\in\NFix(\rho), i+1\in \Fix (\pi)\}.
\end{align}
Then 
$$
P_{\pi,\rho}^{J}=q^{\# (A_{\pi,\rho}\cap J)}(1+q)^{\# (A_{\pi,\rho}\cap (S\setminus J))}
$$
\end{corollary}

For example, let $\pi,\rho\in A_{9}$ defined by $\pi=[2,1,3,6,4,7,5,8,9,10]$ and $\rho=[4,2,3,10,5,6,7,8,1,9]$. By Lemma \ref{L:TranAn} we have that $\pi,\rho$ are boolean elements. Since the descents of $\pi^{-1}=[2,1,3,5,7,4,6,8,9,10]$ are $1,5$ and the descents of $\rho^{-1}=[9,2,3,1,5,6,7,8,10,4]$ are $1,3,9$, then $\pi,\rho$ are both in $A_n^J$ for all $J$ such that $J\cap\{s_1,s_3,s_5,s_9\}=\emptyset$. By \cite[Theorem 2.1.5]{Bjorner2005} we get $\pi\le \rho$ and finally by Corollary \ref{C:An} we have $P_{\pi,\rho}^{J}=0$, if and only if $J\cap\{s_4,s_6,s_8\}\ne\emptyset$. In fact $\Exc(\pi)=\{1,5\}$, $\Exc(\pi^{-1})=\{1,4,6\}$, $\Exc(\rho)=\{3,9\}$, $\Exc(\rho^{-1})=\{8,9\}$, $\Fix(\pi)=\{2,3,7,8,9\}$, $\Fix(\rho)=\emptyset$,  $\NFix(\pi)=\emptyset$ and $\NFix(\rho)=\{2,3,5,6,7,8\}$.
For $J=\{s_2,s_4\}\equiv\{2,4\}$ we have $P_{\pi,\rho}^{J}=q(q+1)$.

For the ordinary Kazhdan--Lusztig polynomials Corollary \ref{C:An} becomes
\begin{corollary}
\label{C:Anclas}
Let $\pi,\rho\in A_n$ be two boolean permutations of $[n+1]$ such that $\pi\le \rho$ in the Bruhat order. Then $P_{\pi,\rho}=(1+q)^{\# A_{\pi,\rho}}$, where $A_{\pi,\rho}$ is defined in (\ref{E:A}).
\end{corollary}

Now we consider the Coxeter group $B_n$. It is easy to check that there are two boolean reflections in $B_n$ which are maximal in the Bruhat order: they are  $s_0s_1\cdots s_{n-1}\cdots s_1s_0$ and $s_{n-1}\cdots s_1s_0s_1\cdots s_{n-1}$ (where $s_0$ is the transposition $(1,-1)$ and $s_i$ is the product $(i,i+1)(-i,-i+1)$ in disjoint cycle notation). In fact, given any boolean word $t$, if there is a letter $s_1$ between two occurrences of $s_0$ then move both elements $s_0$ to the beginning and to the end of $t$ (it is possible since $s_0$ commutes with all other elements) and then manipulate the remainig letters as a subword in $A_{n-1}$; if $s_0$ is between two occurrences of $s_1$ (and therefore there is exact one $s_0$) then necessarily there are no occurrences of $s_{i+1}$ between the two letters $s_{i}$ for all $i\ge 1$, otherwise $t$ is not a reduced word.

\begin{lemma}
\label{L:Bn}
Let $t_1,t_2\in B_n$, $t_1=s_0\cdots s_{n-1}\cdots s_0$, $t_2=s_{n-1}\cdots s_0 \cdots s_{n-1}$. Let $\pi\in B_n$. Then $\pi$ is a boolean element $\pi\le t_1$ if and only if $\#(\vert\pi([i])\vert \cap [i])\ge i-1$ for all $i\le n$ and the only negative elements in the window notation of $\pi$ may be the first entry or the element $-1$.

Moreover, in this case, if $\overline \pi$ is the reduced word of $\pi$, which is a subword of $t_1$, then $\overline \pi(s_i)=1_l$ if $i+1 $ is a top excedence of $\pi$ (if $i=0$ then the window notation of $\pi$ has only one negative entry which is $-1$); $\overline\pi(s_i)=1_r$ if $i+1$ is a top excedance of $\pi^{-1}$ (if $i=0$ then the window notation of $\pi$ has only one negative entry in the first place); $\overline\pi(s_i)=2$ if and only if $\pi(i+1)=\pi(i+1)$ and $\pi([i+1,n])\ne [i+1,n]$ (if $i=0$ then there are excactly two negative entries in the window notation of $\pi$); $\overline \pi(s_i)=0$ if and only if $\pi([i+1,n])=[i+1,n]$ (if $i=0$ then there is no negative element in the window notation of $\pi$).

The permutation $\pi$ is a boolean element $\pi\le t_2$ if and only if $\#(\vert\pi([i])\vert\cap [i])\ge i-1$ and the only negative entry in the window notation of $\pi$ (if it exists) is in the smallest non-fixed element or in the first $m+1$ entries, if $\vert \pi(i)\vert=i$ for all $i\le m$.

Moreover,  in this case, if $\overline \pi$ is a reduced word of $\pi$, which is a subword of $t_2$ then for all $i\ge 1$,  $\overline \pi(s_i)=1_l$ if $i $ is a bottom excedence of $\pi^{-1}$; $\overline\pi(s_i)=1_r$ if $i+1$ is a bottom excedance of $\pi$; $\overline\pi(s_i)=2$ if and only if $\pi(i)=\pi(i)$ and $\pi([i+1,n])\ne [i+1,n]$; $\overline \pi(s_i)=0$ if and only if $\pi([i+1,n])=[i+1,n]$.
\end{lemma}

The proof is essentially the same of that of Lemma \ref{L:TranAn}. We give the Corollary of Theorem \ref{T:KLBoolean} only for ordinary Kazhdan--Lusztig polynomials. The parabolic case could be done as in Corollary \ref{C:An}.

Let $\pi\in B_n$. We set 
\begin{align*}
\Fix(\pi)&=\{i\in [0,n-1]\vert \pi([i+1,n])=[i+1,n]\}\\
\NFix(\pi)&=\{i\in[n-1]\setminus \Fix(\pi)\vert \pi(i)=i \}\cup \{0 \text{ if $\#\{\pi(i)<0\}=2$}\}
\end{align*}
\begin{corollary}
\label{C:BnClas}
Let $\pi,\rho\in B_n$ two boolean elements in $B_n$ such that $\pi\le \rho$ in the Bruhat order. Then the Kazhdan--Lusztig polynomial $P_{\pi,\rho}$ is given by
$$
P_{\pi,\rho}=\left\{  
\begin{array}{ll}
(1+q)^{B_{\pi,\rho}} & \text{if $\pi\le\rho\le t_1$}\\
(1+q)^{B'_{\pi,\rho}} & \text{if $\pi\le \rho \le t_2$}
\end{array}
\right.
$$
where $B_{\pi,\rho}=\{i\in [0,n-1]\vert i,i+1\in \NFix(\rho),i+1\in \Fix(\pi)\}$, $B'_{\pi,\rho}=\{i\in [0,n-1]\vert i,i+1\in \NFix(\rho),i\in \Fix(\pi)\}$.
\end{corollary}

Now we consider the Coxeter group $D_n$. It is easy to check that the unique boolean reflection of length $2n-1$ is  $s_0s_1s_2\cdots s_{n-1}s_{n-2}\cdots s_2s_1s_0$: in fact, let $t$ any  boolean reflection with the same length. Then any reduced word of $t$ contain both occurrences of $s_0$ or $s_1$ outside the occurrences (maybe only one) of $s_2$. Then move, by commutativity, these occurrences to the leftmost and rightmost place. The central part can be identified with an element of $A_{n-1}$ and we can conclude easily.

\begin{lemma}
\label{L:Dn}
Let $\pi\in D_n$. Then $\pi$ is a boolean element if and only if $\#(\vert \pi([i])\vert\cap [i])\ge i-1$ for all $i\le n$ and the only negative elements in the window notation are in the first two columns and in the entries containing $-1,-2$ (if the first two entries are not $\pm 1,\pm 2$ then these have the same sign).

Let $\overline \pi$ be a reduced word of $\pi$, subword of $s_0s_1\cdots s_{n-1}\cdots s_1s_0$, and let $i\ge 2$. Then $\overline \pi(s_i)=1_l$ if $i+1$ is a top excedance of $\pi$; $\overline \pi(s_i)=1_r$ if $i+1$ is a top excedance of $\pi^{-1}$; $\overline \pi(s_i)=2$ if $\pi(i+1)=i+1$ and $\pi([i+1,n])\ne[i+1,n]$; $\overline \pi(s_i)=0$ if $\pi([i+1,n])=[i+1,n]$.
If $i\le 1$ then there is an occurrence of $s_1$ on the right if $\pi(1)\ge 3$ or $\pi(2)\le -3$; there is an occurrence of $s_1$  on the left if $\pi^{-1}(1)\ge 3$ or $\pi^{-1}(2)\le -3$ or $\pi(1,2)\in\{(2,1),(-1,-2)\}$;
there is an occurrence of $s_0$ on the right if $\pi(1)\le -3$ or $\pi(2)\le -3$; there is an occurrence of $s_0$  on the left if $\pi^{-1}(1)\le -3$ or $\pi^{-1}(2)\le -3$ or $\pi(1,2)\in\{(-2,-1),(-1,-2)\}$.
\end{lemma}

\begin{corollary}
Let $\pi,\rho\in D_n$ be two boolean elements such that $\pi\le \rho$. Then the Kazhdan--Lusztig polynomial $P_{\pi,\rho}$ is given by
$$
P_{\pi,\rho}(q)=(1+q)^{D_{\pi,\rho}}(1+2q)^{D'_{\pi,\rho}}
$$ 
where $D_{\pi,\rho}$ is the number of indices $i$ such that $\rho(i)=i$, $\rho(i+1)=i+1$, $\rho([i+2,n])\ne [i+2,n]$ and $\pi([i+2,n])=[i+2,n]$ incremented by $1$ if $\rho(1)=1, \rho(2)<2, \rho(3)\ne 3$ and $\pi((1,2))\ne(-1,-2)$ or $\rho^{-1}(2)<-2,\vert\rho^{-1}(1)\vert=2$  and $\pi([3,n])=[3,n]$ or $\vert \rho^{-1}(1)\vert>2, \rho(2)\in \{-n,\dots,-3,-1,1\}$ and $\pi([3,n])=[3,n]$; $D'_{\pi,\rho}$ is $1$ if $\rho(1)=1$, $\rho(2)<2$, $\rho(3)=3$ and $\pi([3,n])=[3,n]$ and $D'_{\pi,\rho}=0$ in all other cases.
\end{corollary}

\section{Poincar\'e polynomials}
\label{S:Poincare}
Given $v\in W$, let $F_v(q)=\sum_{u\le v} q^{l(v)} P_{u,v}$. It is well known that, if $W $ is  any Weyl or affine Weyl group, $F_v(q)$ is the intersection homology Poincar\'e polynomial of the Schubert variety indexed by $v$ (see \cite{Kazhdan1980}). 
In this section we compute the Poincar\'e polynomial for any boolean element in a Coxeter group whose Coxeter graph is a tree with at most one vertex having more than two adjacent vertices (such groups include all classical Weyl groups).

Let $v\in W$ be a boolean element and consider the diagram of $(\epsilon, \overline v)$. For convenience we will not depict the second row of each column which is always $0$ and we omit all symbols $\times$. We will call it the diagram of $v$.

Let $v$ be a boolean element and let $s$ be the element of $S$ associated to one of the leftmost vertices in the diagram of $v$. We set $F_{v, s\ne }=\sum q^{l(v)} P_{u,v}$ where the sum runs over all elements $u\le v$ such that $\overline u(s)\ne 0$ and $F_{v,s0}=\sum q^{l(v)} P_{u,v}$ where the sum runs over all elements $u\le v$ such that $\overline u(s)= 0$.

Now consider a diagram $d$. Delete all entries equal to $0$ and delete all edges whose left vertex is not a cell containing $2$. Let $d_1,\dots, d_k$ be the connected components that remain. We refer to them as the \emph{essential components} of $d$.

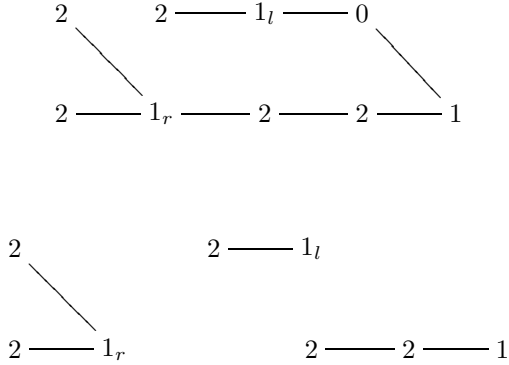
\begin{figure}[hbt]
$$
\xymatrix{ 
2\ar@{-}[dr] & {2}\ar@{-}[r]    & {1_l}\ar@{-}[r] & {0}\ar@{-}[dr]
 \\ {2}\ar@{-}[r]  & {1_r} \ar@{-}[r] & {2} \ar@{-}[r] & {2}\ar@{-}[r] &{1}\\
 \\
}
$$
$$
\xymatrix{
{2}\ar@{-}[dr] & & {2}\ar@{-}[r]    & {1_l} \\
{2}\ar@{-}[r]  & {1_r} & & {2} \ar@{-}[r] & {2}\ar@{-}[r] &{1}
}
$$
\caption{A diagram and its essential components.}
\end{figure}

\begin{lemma}
\label{L:EssentialComponents}
Let $v\in W$ be a boolean element and let $d$ be the diagram of $\overline v$. Let $d_1,\dots, d_k$ be the essential components of the diagram $d$. Let $v_1,\dots, v_k$ be the boolean reflections corresponding to $d_1,\dots, d_k$. Then
$$F_v(q)=\prod_{i=1}^{k}F_{v_i}(q).$$
\end{lemma}
\begin{proof}
We use induction on  $l(v)$. If $l(v)=1$ there is nothing to prove. Now let $l(v)>1$ and let $d_1,\dots,d_k$ be the essential components of $d$ associated to $v$. Let $s$ be the element associated to one leftmost vertex of $v$ and let $d_1$ be the essential component containing such vertex. In this proof we denote by $\overline F_v(q)$  the polynomial corresponding to the diagram $d\setminus d_1$ (known by induction) and by $\widehat F_v(q)$  the polynomial corresponding to $d_1$. If $\overline v(s)=1$  then
by Theorem \ref{T:KLBoolean} or by Lemmas \ref{L:10} and \ref{L:11} and recursion in Proposition \ref{P:KLFromDesc} we have that
$$
F_{v}(q)=(1+q)\overline F_v(q), \quad F_{v,s\ne}(q)=q\overline F_v(q), \quad  F_{v,s0}(q)=\overline F_v(q).
$$
If $\overline v(s)=2$ then we can assume that $d_1$ starts with 
$$ \xymatrix{ {(2)^h}\ar@{-}[r] & {*}},$$
for $h\ge 1$ (otherwise choose another element in $S$). Denote by $s'$  the only element on the right of $s$. By Theorem \ref{T:KLBoolean} and by induction we have
\begin{align*}
F_{v}(q)=& (1+q)^{2h}F'_{v,s'\ne}(q)+(1+q)^h f_{h-\delta}F'_{v,s'0}(q)\\
=&(1+q)^{2h}\widehat F'_{v,s'\ne}(q)\overline F_{v}(q)+(1+q)^h f_{h-\delta}\widehat F'_{v,s'0}(q)\overline F_{v}(q)\\
=&\widehat F'_{v}(q) \overline F_{v}(q),
\end{align*}
where $F'_{v}$ is the polynomial associated to $d$ after deleting all the vertices $(2)^h$ and $\delta$ is determined uniquely by $v$ (and Theorem \ref{T:KLBoolean}). The first factor $(1+q)^{2h}$ denotes the possibility to have all pairs $(2,0)$, $(2,1_l)$, $(2,1_r)$ and $(2,2)$ in the diagram of $(\overline u,\overline v)$ in all $h$ leftmost columns; the second factor $(1+q)^h$ denotes the possibility to have only the pairs $(2,0)$ and $(2,1_l)$.
Similar formulas can be computed for $F_{v,s\ne}(q)$ and $F_{v,s0}$. Therefore we can apply the induction (it is possible that more indices $s\ne$ or $s0$ are necessary; the proof does not change).
\end{proof}

Lemma \ref{L:EssentialComponents} tells us that if we know $F_t(q)$ for all boolean reflection $t$, then it is easy to compute $F_v(q)$ for all boolean elements.

\begin{lemma}
\label{L:22}
Let $v\in W$ be a boolean reflection and suppose that its diagram $d$ has one leftmost vertex $s$ such that if
\xymatrix{ {(2)^h}\ar@{-}[r] & {*}} is a subdiagram of $d$ containing $s$, then necessarily $h=1$. Then
$$
F_{v}(q)=(1+q)^2 F'_{v}(q)  
$$
where $F'_{v}(q)$ is the polynomial associated to diagram $d$ after deleting the vertex $s$.
\end{lemma}

\begin{proof}
By Proposition \ref{P:KLFromDesc}, it is easy to check that
$$
F_{v}(q)=(1+q)^2 F'_{v,s\ne}+ (1+q) F'_{v,s0}(q)(1+q)=(1+q)^2F'_v(q)
$$
where the last factor $(1+q)$ is due to the contribution of 
\begin{tabular}  {|c|c|} \hline $2$ & $2$ \\\hline $*$ & $0$ \\ \hline\end{tabular} in the Kazhdan--Lusztig polynomial $P_{u,v}(q)$ according to Theorem \ref{T:KLBoolean}.
\end{proof}

As corollary of Lemmas \ref{L:EssentialComponents} and \ref{L:22} we have the following result due to Marietti \cite[Theorem 8.1]{Marietti2006}
\begin{corollary}
\label{C:Marietti}
Let $v\in S_{n+1}$ be a boolean element. Let $t$ be the boolean reflection $s_1\cdots s_n\cdots s_1$ with $s_i$ be the transposition $(i,i+1)$. Let $\overline v$ be the reduced word of $v$ subword of $t$. Then
$$
F_{v}(q)=(1+q)^{l(v)-2a(v)}(1+q+q^2)^{a(v)},
$$
where $a(v)$ is the number of patterns $(2,1_*)$ in $\overline v$.
\end{corollary}
By Lemmas \ref{L:EssentialComponents} and \ref{L:22} its proof reduces to compute $F_{s_1s_2s_1}(q)=(q^2+q+1)(1+q)$ and $F_{s_1}=(1+q)$ in $S_3$.

To prove the next result we have to compute the polynomials $F_{v,s\ne}(q)$ and $F_{v,s0}(q)$ with $v$ associated to the diagram 
\xymatrix{
2 \ar@{-}[r] & 2\ar@{-}[r] & \dots \ar@{-}[r] &1
}
with $i$ vertices. Let $s\in S$ be the element corresponding  to the first vertex and let $s'\in S$ be the element associated to the second vertex.
If $i=2$ then by direct computation we have
$$
F_{v,s\ne}(q)=q(1+q)^2 \quad  F_{v,s0}(q)=(1+q).
$$
By induction it is easy to compute that
\begin{align}
F_{v,s\ne}=&(2q+q^2)F'_{v,s'\ne}(q)+q F'_{v,s'0}(q)(1+q) = q(1+q)^{2i-2}\notag\\
F_{v,s0}=& F'_{v,s'\ne}(q)+ F'_{v,s'0}(q)(1+q)=(1+q)^{2i-3},\label{E:Fe1}
\end{align}
where $F'_{v}(q)$ denotes, as usual, the polynomial associated to the diagram without the first vertex. Similarly, let $v $ be the boolean reflection corresponding to the diagram
$$
\xymatrix{
2 \ar@{-}[r] & 2\ar@{-}[r] & \dots & 2\ar@{-}[r] \ar@{-}[r] &1\\
             &             &       & 2 \ar@{-}[ur]
}
$$
with $i+1$ vertices.
Then 
\begin{equation}
\label{E:Fe2}
F_{v,s\ne}=q(1+q)^{2i} \text{ and } F_{v,s0}=(1+q)^{2i-1}.
\end{equation}

\begin{proposition}
\label{P:Poincare}
Let $W$ be a Coxeter group such that its Coxeter graph is a tree and all vertices except at most one have degree less than $3$. Denote with $w$ such exceptional vertex. Let $v\in W$ be a boolean element. Then
$$
F_{v}(q)=(1+q+q^2)^{k-1}\big(q(1+q)^{h+1}+f_h(q)\big)(1+q)^{l(v)-2k-h-2},
$$
where $k$ is the number of essential components of the diagram $d$ of $v$ with at least two vertices and  $h$ is the number of entries equal to $2$ in the adjacent cells of $w$ (also consider  the cell on the right).
\end{proposition}
The formula is also true when there is no vertex of degree greater than $2$: in this case let $w$ be any vertex of degree $2$. 

\begin{proof}
By Lemma \ref{L:EssentialComponents} it suffices to compute the polynomial associated to the only non-trivial component.  By Lemma \ref{L:22} it suffices to consider only the following two cases.
$$
\xymatrix{
{(2)^{h'}}\ar@{-}[r] & \dots\ar@{-}[r] & 1
}$$
$$
\xymatrix{
{(2)^{h'}}\ar@{-}[r] & \dots\ar@{-}[r] & 1\\
& 2\ar@{-}[ur]
}.$$

In the first case we compute
\begin{align*}
F_{v}(q)=&(1+q)^{2h'} F'_{v,s'\ne}(q)+(1+q)^{h'} f_{h}F'_{v,s'0}(q)\\
=& (1+q)^{h'+2i-3}\big(q(1+q)^{h+1}+f_h\big)\quad\quad \text{ by (\ref{E:Fe1})},
\end{align*}
where $F'_{v}(q)$ is the polynomial associate to the diagram without the $h'$ leftmost cells and $i$ is an integer.
The second case is similar; use (\ref{E:Fe2}).
\end{proof}

\end{document}